\documentclass{amsart} 

\usepackage{amsmath, amssymb,amsfonts,xspace,epsfig,stackrel,latexsym,wrapfig,amsthm,
verbatim,psfrag,graphicx,rotating,booktabs}
\RequirePackage[numbers]{natbib}

\RequirePackage[colorlinks,citecolor=blue,urlcolor=blue]{hyperref}

\newtheorem{theorem}{Theorem}[section]
\newtheorem{lemma}[theorem]{Lemma}

\theoremstyle{definition}
\newtheorem{definition}[theorem]{Definition}
\newtheorem{example}[theorem]{Example}

\theoremstyle{remark}
\newtheorem{remark}[theorem]{Remark}

\numberwithin{equation}{section}

\newcommand\expP[1]{\exp\Bigl(#1\Bigr)}
\newcommand\ceil[1]{\lceil#1\rceil}
\newcommand\floor[1]{\lfloor#1\rfloor}

\newcommand{\set}[2]{\left\{ #1  \;\middle\vert\; #2 \right\}}

\newcommand{\twopiece}[5][0]{
    \ifcase#1
        \left\{\begin{array}{ll}{#2}&{\text{ if } #3}\\{#4}&{\text{ if } #5}\end{array}\right.
    \else
        \left\{\begin{array}{ll}{#2}&{\text{ if } #3}\vspace{#1pt}\\{#4}&{\text{ if } #5}\end{array}\right.
\fi}

\newcommand{\threepiece}[7][0]{ 
    \ifcase#1
        \left\{\begin{array}{ll}{#2}&{\text{ if } #3}\\{#4}&{\text{ if } #5}\\{#6}&{\text{ if } #7}\end{array}\right.
    \else
        \left\{\begin{array}{ll}{#2}&{\text{ if } #3}\vspace{#1pt}\\{#4}&{\text{ if } #5}\vspace{#1pt}\\{#6}&{\text{ if } #7}\end{array}\right.
\fi}

\def\N{\mathbb{N}}
\def\P{\mathbb{P}}

\def\R{\mathbb{R}}

\newcommand{\cC}{{\mathcal C}}
\newcommand{\cD}{{\mathcal D}}

\newcommand{\cF}{{\mathcal F}}

\newcommand{\cI}{{\mathcal I}}
\newcommand{\cJ}{{\mathcal J}}
\newcommand{\cK}{{\mathcal K}}

\newcommand{\cM}{{\mathcal M}}
\newcommand{\cN}{{\mathcal N}}

\newcommand{\eps}{\varepsilon}

\begin{document}

\title[Entropy, Convex Functions]{Entropy of Convex Functions on $\R^d$}


\author{Fuchang Gao}
\address{Department of Mathematics, University of Idaho, Moscow, ID 83844-1103, USA}
\curraddr{}
\email{fuchang@uidaho.edu}
\thanks{Research partially supported by a grant from the Simons Foundation, \#246211.}

\author{Jon A. Wellner}
\address{Department of Statistics, University of Washington, Seattle, WA  98195-4322, USA}
\curraddr{}
\email{jaw@stat.washington.edu}
\thanks{Supported in part by NSF Grants DMS-1104832 and DMS-1566514, and by NI-AID grant 2R01 AI291968-04.}

\subjclass[2010]{Primary 52A41, 41A46; Secondary 52A27,  52C17, 52B11}
\keywords{metric entropy, bracketing entropy, convex functions, polytopes, simplicial approximation}
\date{}

\dedicatory{}

\begin{abstract}
Let $\Omega$ be a bounded closed
convex set in $\R^d$ with non-empty interior, and let $\cC_r(\Omega)$
be the class of convex functions on $\Omega$ with $L^r$-norm bounded by 1.
We obtain sharp estimates of the
$\eps$-entropy of $\cC_r(\Omega)$ under $L^p(\Omega)$ metrics, $1\le p<r\le \infty$.
In particular, the results imply that the universal lower bound $\eps^{-d/2}$ is also an
upper bound for all $d$-polytopes, and the universal upper bound of
$\eps^{-\frac{(d-1)}{2}\cdot \frac{pr}{r-p}}$ for $p>\frac{dr}{d+(d-1)r}$ is attained by the closed
unit ball. While a general convex body can be approximated
by inscribed polytopes, the entropy rate does not carry
over to the limiting body.
Our results have
applications to questions concerning rates of convergence of nonparametric estimators
of high-dimensional shape-constrained functions.
\end{abstract}

\maketitle

\bigskip

\section{Introduction}
Given a set $T$ in a metric space $(X,\rho)$, the $\eps$-covering number of $T$,
denoted by $N(\eps,T, \rho)$, is the minimum number of closed balls of radius
$\eps$ in $(X,\rho)$ needed to cover $T$. It is a measurement of massiveness of
$T$ at a fixed resolution $\eps$. With a varying radius, the covering number quantitatively gauges the geometric complexity of $T$.
Over half a century ago, Kolmogorov and Tihomirov \cite{MR0124720} 
put the study of the logarithm of covering number with varying radius, or metric entropy, at the center stage.
Since then, metric entropy has come to play an increasingly important
role in a wide range of problems in mathematics including approximation theory, probability theory, information theory
and statistics. In particular, it is now widely understood that accurate bounds for metric entropy determine optimal
rates of convergence in estimation problems in statistics;  see, for example,
\cite{MR0334381}, 
\cite{MR722129},  
\cite{MR1742500}, 
and
\cite{MR1240719}. 

The focus in this paper is on metric entropy for various classes of convex functions. In particular we study metric entropy of the classes
$\cC_r (\Omega)$ of all real-valued convex functions $f$ on a closed convex body $\Omega$ in
$\R^d$ having $L^r$-norm bounded by $1$. These convex functions are of special importance not only because they are basic classes of functions, but also because they appear so commonly in applications. For example, exponential functions that frequently appear in statistical density estimation are convex. In these statistical applications, one often also needs to know the so-called bracketing entropy, that is, the logarithm of the minimum number
$N_{[\,]}(\eps, \cC_\infty(\Omega), \|\cdot\|_p)$ of
$\eps$-brackets
$$
[\underline{f},\overline{f}]:=\set{g\in \cC_\infty(\Omega)}{\underline{f}\le g\le \overline{f}}, \ \|\overline{f}-\underline{f}\|_p\le \eps
$$
needed to cover $\cC_\infty(\Omega)$. It is known and easy to see that
$$
N(\eps,\cC_\infty(\Omega), \|\cdot\|_p)\le N_{[\,]}(2\eps,\cC_\infty(\Omega), \|\cdot\|_p).
$$

Before we present our results, let us briefly review the history of metric entropy
bounds for convex functions and several recent uses of such bounds:

For the class $\cC$ of all compact convex subsets of a fixed bounded subset $\Omega$
of $\R^d$ endowed with the Hausdorff metric $h$,
Bronshtein \cite{MR0415155}
obtained both upper and lower bounds of the order $\eps^{- (d-1)/2}$
for the metric entropy $\log N(\eps , \cC,  h)$. In the same paper, Bronshtein
also obtained bounds of the order $\eps^{-d/2}$  for $\log N(\eps, \cF, \| \cdot \|_{\infty} )$ where
$\cF$ is the class of all convex functions $f$ defined on a fixed convex body $\Omega$ in $\R^d$
satisfying a (uniform) Lipschitz condition:  $| f(y) - f(x) | \le L \| y - x \|$ for all $x,y \in \Omega$; here
$\| \cdot \|_{\infty}$ denotes the supremum norm.
These bounds improved  earlier results of
Dudley \cite{MR0358168}, and are incorporated in
\cite{MR876079, MR1720712, MR3445285}.

In the case $d=1$,
Gao \cite{MR2386068} 
removed the requirement of uniform Lipschitz condition, and obtained sharp bounds for $\log N(\eps, \cC_{\infty}([a,b]), L_2 )$,
and, in fact, provided upper and lower bounds of the order $\eps^{-1/k}$  for the ``$k$-monotone'' classes
$$
\cM_k ( [a,b]):= \{ f : [a,b] \rightarrow [-1,1] \big | \ (-1)^{i} f^{(i)} (x) \ge 0, f(a)=0\ 1 \le i \le k, \ x \in [a,b]\}.
$$
When $k=2$, it reduces to the convex case. These results were further extended for
$\log N(\eps, \cM_k ([a,b]), L_p )$ and $\log N_{[\,]}(\eps, \cM_k ( [a,b]), L_p )$ in
Gao and Wellner \cite{MR2520591}. The results for $\log N(\eps, \cC_{\infty}([a,b]), L_p)$
were also obtained independently by Dryanov \cite{MR2519658}.  

For the case $d>1$,
Guntuboyina and Sen \cite{MR3043776}
extended the result for $\log N(\eps, \cC_{\infty}([a,b]^d), L_p)$. More recently,
Guntuboyina \cite{MR3474567}  
relaxed the restriction on the uniform norm by considering classes $\cC_r ([a,b]^d)$.
He showed that the metric entropy $\log N(\eps, \cC_r ([a,b]^d), L_p)$ are infinite for $p\ge r$
(since these classes are not precompact in $L_p$ with $p\ge r$), and are of the
order $\eps^{-d/2}$ if $p < r$.

In this paper, we study the metric entropy of $\cC_r(\Omega)$ under $L^p$-norm, $1\le p<r$, for
all compact convex sets $\Omega$ in $\R^d$ with non-empty interior.
It turns out that the growth rate of the metric entropy of $\cC_r(\Omega)$ heavily depends on the shape of the
convex domain $\Omega$. This heavy dependence on shapes makes the problem
both more interesting and challenging.

Now, we turn to the statements of our results.

We first show that $\eps^{-d/2}$ is the general lower bound for the metric entropy rate, and if $\Omega$
is a closed convex polytope, then $\eps^{-d/2}$ is also the upper bound.
More precisely, we will prove the following theorem. Throughout the rest of the paper, $|\Omega|$ stands for the Lebesgue measure of $\Omega$.

\begin{theorem}\label{thm1}
Let $\Omega$ be a
compact
convex set in $\R^d$ with non-empty interior.

(i) There exists a constant
$c_1$ depending only on $d$ such that for all $\eps>0$,
\begin{align*}
\log N (\eps, \cC_r(\Omega), \| \cdot \|_p ) \ge c_1|\Omega|^{\frac{d}{2p}-\frac{d}{2r}}\eps^{-d/2}.
\end{align*}

(ii) If $\Omega$ can be triangulated into $m$ simplices of dimension $d$,
then for any $1 \le p < r$,
there exists a constant
$C_1$ depending on $p, d, r$,
such that for any $\eps>0$,
\begin{align*}
&\log N (\eps, \cC_r(\Omega), \| \cdot \|_p)
\le C_1 m |\Omega|^{\frac{d}{2p}-\frac{d}{2r}}\eps^{-d/2}.
\end{align*}
Consequently, if $\Omega$ is a convex polytope with $v$ extreme points,
then we can choose $m=O(v^{\ceil{\frac{d}2}})$.
When $r=\infty$, the same inequality holds for bracketing entropy.
\end{theorem}

In view of the fact that a general compact convex set
can be approximated by convex sets
with finitely many extreme points, one might guess that the rate $\eps^{-d/2}$ holds for
general compact convex sets in $\R^d$ with non-empty interior.
This, however, is not the case. This is because the upper bound increases as
$m$ increases. This dependence on $m$ is important.
It enables us to establish upper bounds for general bounded
convex sets. For that, we need the following definition.
\begin{definition}
Let $\Omega$ be a bounded closed convex set in $\R^d$ with non-empty interior.
A sequence of (non-degenerate) $d$-simplices $\cD=\{ D_1, D_2, \ldots \}$ is called a
{\sl simplicial approximation sequence} for $\Omega$ if $D_i \subset \Omega$
for all $i\in \N$ and $D_i^\circ \cap D_j^\circ =\emptyset$ for all $i \not= j $ (where
$D^\circ$ denotes the interior of $D$).
For $t \in (0,1)$
we define
$$
S_{\cD}(t, \Omega)=
\min \{ j\in \N : \ | \Omega \setminus \cup_{i \le j} D_i | \le t | \Omega | \},
$$
and we call $S_{\cD} (t, \Omega)$ the {\sl simplicial approximation number of $\Omega$ according to $\cD$}.
\end{definition}

\begin{example}\label{example}
As an example to illustrate simplicial approximation sequences and simplicial approximation,
we consider the case when $\Omega$ is the closed unit disk in $\R^2$.
We choose $D_1$ as an inscribed equilateral triangle. For each edge of $D_1$,
we build an isosceles triangle with apex on the short arc opposite to the edge.
Denote these three isosceles triangles by $D_2, D_3, D_4$. The union of $D_1,\cdots D_4$
is a regular hexagon inscribed in the disk. Now, on each edge of the hexagon, we build an isosceles
triangle with apex on the short arc opposite to the edge, and denote these six isosceles triangles
by $D_5, D_6, \ldots, D_{10}$. The union of $D_1, D_2,\ldots, D_{10}$ is a regular 12-gon inscribed in the disk
(see Figure 1). Continuing this process, we obtain a simplicial approximation sequence,
$\cD=\{D_1,D_2,D_3,\ldots\}$. It is not difficult to see that for all
$n=0, 1,\ldots$ and $k=1, 2,\ldots, 3\cdot 2^n$,
$S_{\cD}(t, \Omega)=1$ if $1-\frac{3}{2\pi}\sin\frac{2\pi}{3}\le t\le 1$, and
$$
S_{\cD}(t, \Omega)=3\cdot 2^{n}-2+k, \qquad \mbox{for} \ \  a_n-k b_n\le t< a_n-(k-1)b_n,
$$
where
$$
a_n=1-\frac{3\cdot 2^{n-1}}{\pi}\sin\frac{\pi}{3\cdot 2^{n-1}}, \ \ \   b_n=\frac1{\pi}\left(1-\cos\frac{\pi}{3\cdot 2^n}\right)\sin\frac{\pi}{3\cdot 2^{n}},
$$
from which we conclude that $S_{\cD}(t,\Omega)=O(t^{-1/2})$.
\begin{figure}[h!]
\begin{center}
\begin{minipage}{4in}
  \centerline{\includegraphics[height=2in]{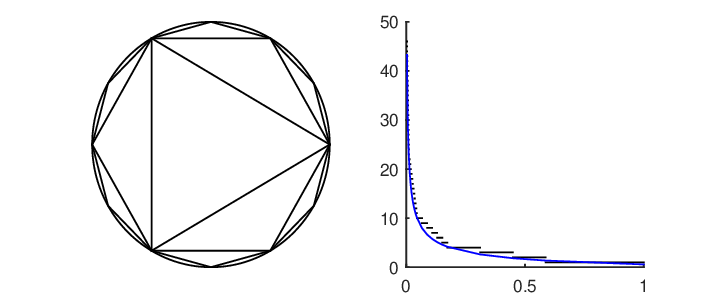}}
  \caption[fontsize=small]{Left: A simplicial approximation sequence for the unit
  disk in $\R^2$. Right: The graph of the function $S_{\cD}(t,\Omega)$ in black,
  compared with that of the function $f(t)=\frac{2\pi}{\sqrt{6}}t^{-1/2} -2$ in blue.}
\end{minipage}
\end{center}
\end{figure}
\end{example}

Now we can state the following theorem.
\begin{theorem}\label{thm2}
Let $\Omega$ be a
compact
convex set in $\R^d$ with non-empty interior.
Let $\cC_r(\Omega)$ be the set of convex functions on $\Omega$ whose
$L^r(\Omega)$-norms are bounded by $1$.
Then, there exists a constant $C$ depending only on $d$, $p$ and $r$, such that for any $0<\eps<1$, and any simplicial approximation sequence $\cD$,
\begin{eqnarray*}
\log N(\eps |\Omega|^{\frac1p-\frac1r},\cC_r(\Omega), \|\cdot\|_p)
& \le & C\int_{\delta(\eps)}^{1} \frac{S_{\cD}(t,\Omega)}{t}dt \\
&& \ \ \ + \ C\left(\int_{\delta(\eps)}^{1} \left(\frac{S_{\cD}(t,\Omega)}{t}\right)^{\beta}dt\right)^{1/\beta}\cdot\eps^{-d/2}
\end{eqnarray*}
where $\delta(\eps)=2^{-2-\frac{r}{r-p}}\eps^{\frac{rp}{r-p}}$
and $\beta=\frac{2pr}{2pr+(r-p)d}$. When $r=\infty$, the same inequality holds for bracketing entropy.
\end{theorem}

By specifically constructing a simplicial approximation for the ball, we show that Theorem~\ref{thm2} implies the following theorem. Our proof of the theorem also provides a general scheme of constructing simplicial approximations for a given convex set.
\begin{theorem}\label{thm3}
If $\Omega$ is a compact convex set contained in the closed unit ball in $\R^d$, then there exists a
constant $C$ depending only on $r$, $p$ and $d$ such that for all $1\le p<r\le \infty$ and all $0<\eps<1$,
$$
\log N(\eps,\cC_r(\Omega), \|\cdot\|_p)
\le
\threepiece{C\eps^{-\frac{(d-1)}{2}\cdot \frac{pr}{r-p}}}{p>\frac{dr}{d+(d-1)r}}
{C\eps^{-d/2}\left|\log \eps\right|^{1+\frac{(r-p)d}{2pr}}}{p=\frac{dr}{d+(d-1)r}}
{C\eps^{-d/2}}{p<\frac{dr}{d+(d-1)r}}.
$$
When $r=\infty$, the same inequality holds for bracketing entropy.
\end{theorem}

The following theorem implies the sharpness of Theorem~\ref{thm3}
at least for the case when $\Omega$ is the closed unit ball in $\R^d$, $r=\infty$
and $p\ne \frac{d}{d-1}$.
\begin{theorem}\label{thm4}
If $\Omega$ is the closed unit ball in $\R^d$, then there exists a constant
$c_2$ dependent only on $d$ and $p$
such that for all $0<\eps < 1$,
$$
\log N(\eps, \cC_\infty(\Omega), \|\cdot\|_p) \ge c_2 \eps^{-\gamma},
$$
where $\gamma=\max\{(d-1)p/2, d/2\}$.
\end{theorem}

\begin{remark}
Because Theorem~\ref{thm3} is built upon Theorem~\ref{thm2}, which is again based
on Theorem~\ref{thm1} (ii), Theorem~\ref{thm4} indicates that in some cases, the linear
dependence on $m$ in the upper bound in Theorem~\ref{thm1} (ii) is optimal.
A more concrete example is the regular $(m+2)$-gon in $\R^2$. By the end of this paper,
we will show that for this $\Omega$, there exists a constant $c_2$ such that for $0<\eps\le \frac14 m^{-2}$,
\begin{align}
\log N (\eps, \cC_r(\Omega), \| \cdot \|_p)
\ge c_2 m \eps^{-1}.\label{linear}
\end{align}
\end{remark}
In general, however, the lower bound should depend on the geometry of the set $\Omega$,
and cannot be simply captured by the minimum number of $d$-simplices required in a triangulation.
\section{Proofs}
\subsection{Scaling}
In this subsection, we prove two lemmas, through which we can reduce a problem
on an arbitrary closed convex set with non-empty interior to a problem on a closed
convex set contained in $[0,1]^d$ with volume at least $1/d!$.

\begin{lemma}[Boxing a Convex Set]\label{box}
Every compact convex set $\Omega$ in $\R^d$ with a non-empty interior
can be enclosed in a closed rectangular box of volume $d!|\Omega|$,
and contains a convex polytope of at most $2d$ vertices and volume at least $|\Omega|/d!$,
where $|\Omega|$ stands for the Lebesgue measure of $\Omega$.
\end{lemma}

\begin{proof}
We use induction on $d$ to show that we can find positive
numbers $h_1, h_2, \ldots, h_d$ such that $\Omega$ is contained in a
rectangular box of size $h_1\times h_2\times \cdots\times h_d$, and contains a
convex polytope of at most $2d$ vertices with volume at least $\frac1{d!}\cdot h_1\times h_2\times \cdots\times h_d$.

The statement is trivial if $d=1$. Suppose the statement is true for $d=k$.
Consider the case $d=k+1$. Let $h_{k+1}={\rm diam}(\Omega)$.
Choose $x,y\in \Omega$ so that $\|x-y\|_2=h_{k+1}$. Let $P_x^\perp(\Omega)$
be the projection of $\Omega$ onto the affine hyperplane that contains
$x$ and is orthogonal to $x-y$. Since $P_x^\perp (\Omega)\subset \R^k$ is a
$k$-dimensional compact convex set with non-empty interior, by the induction
hypothesis, we can find positive numbers $h_1, h_2, \ldots, h_k$ such that
$P_x^\perp(\Omega)$ is contained in a rectangular box $R_k$ of size $h_1\times h_2\times \cdots\times h_d$,
and contains a convex polytope $T_k$ of at most $2k$ vertices with volume at least
$\frac1{k!}\cdot  h_1\times h_2\times \cdots\times h_k$.
If we let $[x,y]$ be the line segment between $x$ and $y$, then $\Omega$ is
clearly contained in the rectangular box $R_k\times [x,y]$ of
size $h_1\times h_2\times \cdots\times h_{k+1}$.

To show that $\Omega$ contains a convex polytope of at most $2(k+1)$ vertices with volume at least
$\frac{1}{(k+1)!}\cdot h_1\times h_2\times \cdots\times h_{k+1}$, we let
$u_1,u_2,\ldots u_{m}$, $m\le 2k$, be the vertices of the convex polytope $T_k$.
Clearly, the convex hull $U$ of $\{x,y,u_1,u_2,\ldots, u_m\}$ has volume
$|U|\ge \frac{1}{(k+1)!}\cdot h_1\times h_2\times \cdots\times h_{k+1}$.

For each $1\le i\le m$, there exists $z_i\in \Omega$ such that $P_x^\perp z_i=u_i$.
Because $\Omega$ is convex, it contains the convex hull of
$\{x,y, z_1,z_2,\cdots, z_{m}\}$. Denote this convex hull by $T_{k+1}$. Then $T_{k+1}$ has at most $2(k+1)$ vertices.
Note that the volume of $T_{k+1}$ is at least as large as $|U|$. Indeed, for any
unit vector $u$ perpendicular to $x-y$, consider the half-line in the direction of $u$
starting from $x$. Suppose the half-line intersects the boundary $U$ at $w(u)$.
Choose $z(u)\in T_{k+1}$ such that $P_x^\perp z(u)=w(u)$. Clearly, the area of
$\triangle xyz(u)$ is the same as that of $\triangle xyw(u)$, which equals
$\frac12 h_{h+1}\|x-w(u)\|_2$.
Let $\sigma_{k-1}$ be the $(k-1)$-dimensional spherical measure on the $(k-1)$-dimensional unit sphere $S^{k-1}$.
By using a cylindrical system to compute the volume of $T_{k+1}$, we have
$$
|T_{k+1}|\ge \int_{S^{k-1}}{\rm area }(\triangle xyz(u))\,d\sigma_{k-1}(u)
=\int_{S^{k-1}}{\rm area }(\triangle xyw(u))\,d\sigma_{k-1}(u)=|U|.
$$
Hence, $|T_{k+1}|\ge \frac{1}{(k+1)!}\cdot h_1\times h_2\times \cdots\times h_{k+1}$.
This proves the case $d=k+1$, and thus the statement at the beginning of the proof,
which implies that the volume of $\Omega$ is at least $\frac1{d!}$ of that of the rectangular box.
Hence, the volume of the rectangular box is bounded by $d!|\Omega|$.
\end{proof}

\begin{lemma}[Scaling]\label{scaling}
Let $\Omega$ be a bounded closed convex set contained in a closed rectangular box
$R$ with volume $|R|$ , and let $T$ be any affine transform that maps $R$ onto $[0,1]^d$.
Then for all
$1\le p<r<\infty$ and $\eps>0$,
\begin{align}
N(\eps, \cC_r(\Omega), \| \cdot \|_{L^p (\Omega)} )
= N( |R|^{\frac1r-\frac1p} \eps , \cC_r(T(\Omega)), \|\cdot\|_{L^p(T(\Omega))}) .
\label{ScalingIdentity}
\end{align}
Similarly, for all {$1\le p<\infty$ and $\eps>0$},
\begin{align}
N_{[\, ]} (\eps, \cC_\infty(\Omega), \| \cdot \|_{L^p (\Omega)} )
= N_{[\, ]} ( |R|^{-\frac1p} \eps , \cC_\infty(T(\Omega)), \|\cdot\|_{L^p(T(\Omega))}).
\label{ScalingIdentity2}
\end{align}
\end{lemma}

\begin{proof}
Let $f \in {\cC}_r (T(\Omega))$.  Then $| R |^{-1/r} f \circ T \in {\cC}_r (\Omega)$ since
$$
\int_{\Omega} |f|^r \circ T d \lambda
= \int_{T(\Omega)} | f |^r | T^{-1} | d \lambda = | R | \int_{T(\Omega)}  | f |^r d \lambda \le  | R |,
$$
where $\lambda$ is Lebesgue measure on $\R^d$.
Now let $f_1, \ldots , f_N$ be an $L_p(\Omega)$ $\eps$-net for ${\cC}_r (\Omega)$.
Then, for $f \in {\cC}_r (T(\Omega))$ we have
\begin{eqnarray*}
\left ( \int_{T(\Omega)} \big | f - | R |^{1/r} f_i \circ T^{-1} \big |^p d \lambda \right )^{1/p}
& = &  \left ( \int_{\Omega} \big | f\circ T  - | R |^{1/r} f_i  \big |^p | T | d \lambda \right )^{1/p} \\
& = &  \left ( \int_{\Omega}\big | | R |^{1/r}  \left ( | R|^{-1/r} f\circ T  -  f_i \right ) \big |^p |R|^{-1}  d \lambda \right )^{1/p} \\
& = & | R |^{1/r - 1/p}  \left ( \int_{\Omega} \big |  | R|^{-1/r} f\circ T  -  f_i  \big |^p   d \lambda \right )^{1/p},
\end{eqnarray*}
and since $|R|^{-1/r} f\circ T \in {\cC}_r (\Omega)$, for some $i \in \{ 1, \ldots , N \}$ the last display is bounded
above by $|R|^{1/r - 1/p} \eps$.
Thus given an $L_p (\Omega)$ $\eps$-net for ${\cC}_r (\Omega)$ we have constructed an $L_p (T(\Omega))$
$|R|^{1/p-1/r} \eps$-net for ${\cC}_r (T(\Omega))$.
It follows that
\begin{eqnarray}
 N( |R|^{1/r - 1/p} \eps , {\cC}_r (T(\Omega)), \| \cdot \|_{L_p (T(\Omega))})
\le N(\eps, {\cC}_r (\Omega), \| \cdot \|_{L_p (\Omega)}  ) .
\label{CoveringNumberScalingInequalityOne}
\end{eqnarray}
By a similar argument we find that
\begin{eqnarray}
 N( |R|^{1/r - 1/p} \eps , {\cC}_r (T(\Omega)), \| \cdot \|_{L_p (T(\Omega))})
\ge N(\eps, {\cC}_r (\Omega), \| \cdot \|_{L_p (\Omega)}  ) ,
\label{CoveringNumberScalingInequalityTwo}
\end{eqnarray}
and hence the equality (\ref{ScalingIdentity}) holds.

To prove (\ref{ScalingIdentity2}), first note that if $f \in {\cC}_{\infty} (T(\Omega))$,
then $\sup_{\Omega} | f \circ T | = \sup_{T(\Omega)} | f| \le 1$, so $ f \circ T \in {\cC}_{\infty} (\Omega)$.
Then suppose that $[\underline{f}_i , \overline{f}_i ]$, $1 \le i \le N$, are $L_p (\Omega)$  brackets of size $\eps$
for ${\cC}_{\infty}(\Omega)$.
Then $[\underline{f}_i \circ T^{-1} , \overline{f}_i  \circ T^{-1}]$, $1 \le i \le N$, are $L_p (T(\Omega))$ brackets of
size $| R |^{-1/p} \eps$ for ${\cC}_{\infty} (T(\Omega))$.  To see this, note that for some $i \in \{ 1, \ldots , N \}$
$$
\underline{f}_i (x) \le f \circ T (x) \le \overline{f}_i (x) \ \ \ \mbox{for all} \ \ x \in \Omega,
$$
and hence
$$
\underline{f}_i \circ T^{-1}(y) \le f (y) \le \overline{f}_i \circ T^{-1}(y) \ \ \ \mbox{for all} \ \ y \in T(\Omega).
$$
Furthermore,
\begin{eqnarray*}
\lefteqn{\int_{T(\Omega)} | \overline{f}_i \circ T^{-1} (y) - \underline{f}_i \circ T^{-1} (y) |^p d \lambda
 =  \int _{\Omega} | \overline{f}_i  - \underline{f}_i |^p | T | d \lambda } \\
& = & | R |^{-1} \int_{\Omega}   | \overline{f}_i  - \underline{f}_i |^p d \lambda   \le  \left ( | R |^{-1/p} \eps \right )^p .
\end{eqnarray*}
Thus
\begin{eqnarray*}
 N_{[ \, ]} ( |R|^{ - 1/p} \eps , {\cC}_\infty (T(\Omega)), \| \cdot \|_{L_p (T(\Omega)})
\le N_{[\, ]} (\eps, {\cC}_\infty (\Omega), \| \cdot \|_{L_p (\Omega)}  ) .
\end{eqnarray*}
A similar argument yields the reversed inequality, and hence (\ref{ScalingIdentity2}) holds.
\end{proof}

By combining Lemma~\ref{box} and Lemma~\ref{scaling}, we have
\begin{align}
N(\eps, \cC_r(\Omega), \| \cdot \|_{L^p (\Omega)} )
\le  N((d!)^{\frac1r-\frac1p}\cdot |\Omega|^{\frac1r-\frac1p} \eps , \cC_r(T(\Omega)), \|\cdot\|_{L^p(T(\Omega))}),
\label{rescaling}
\end{align}
and
\begin{align}
N_{[\, ]} (\eps, \cC_\infty(\Omega), \| \cdot \|_{L^p (\Omega)} )
\le  N_{[\, ]} ((d!)^{-\frac1p}\cdot |\Omega|^{-\frac1p} \eps , \cC_\infty(T(\Omega)), \|\cdot\|_{L^p(T(\Omega))}),
\label{rescaling2}
\end{align}
where $T(\Omega)$ has volume at least $1/d!$ and is contained in $[0,1]^d$.

\subsection{Under Uniform Lipschitz}
In this subsection, we recall that if we assume the functions in $\cC_r(\Omega)$ are bounded and
uniform Lipschitz, then the metric entropy estimate would follow
from the following known results of Bronshtein
\cite{MR0415155}.  

\begin{lemma}[Bronshtein]\label{Bronshtein} Let $\cK(\rho)$ be the set of all
closed convex sets contained in the closed Euclidean
ball of radius $\rho$ in $\R^{d+1}$, $d\ge1$.  Let $h$ be the Hausdorff distance on $\cK(\rho)$.
There exists a constant $C_0$ depending only on $d$, such that for any $0<\eps < \rho$,
$$
\log N(\eps , \cK(\rho), h ) \le C_0 (\rho \eps^{-1} )^{d/2}.
$$
\end{lemma}

\begin{lemma}
\label{uniform-Lipschitz}
Let $\Omega$ be a closed convex set in $[0,1]^d$, and let $\cF_{\alpha} (\Omega)$ be the class of
convex functions on $\Omega$ that are bounded by $M$ and have Lipschitz constant bounded by $\alpha$.
Then for all $\eps< 2^{-1-1/p} \sqrt{(1+\alpha^2)(M^2+d/4)}$,
\begin{eqnarray*}
\log N_{[\, ]} (\eps, \cF_{\alpha} (\Omega), \| \cdot \|_{L^p (\Omega)} )
\le 2^{-d}C_0 \{(1 + \alpha^2) (4M^2+d)\}^{d/4} \eps^{-d/2}
\end{eqnarray*}
where $C_0$ is the same constant as in Lemma~\ref{Bronshtein}.
\end{lemma}

\begin{remark}  Lemma~\ref{Bronshtein} can be found in
\cite{MR0415155}, 
\cite{MR1720712},
or 
\cite{MR1385671}
(Lemma 2.7.8, page 163). 
Lemma~\ref{uniform-Lipschitz} is also known for regular metric entropy.
For example, it would follow from
\cite{MR1385671}
Corollary 2.7.10, page 164. Because we deal with bracketing entropy,
 we include a proof here for the convenience of the reader.
\end{remark}

\begin{proof}
For each $f \in \cF_{\alpha} (\Omega)$, since $\Omega$ is a closed and convex set and $f$ is convex, the
epigraph $\mbox{epi}(f):= \{ (x,t) : \ f(x) \le t \le M, \ x \in \Omega \}$ is a closed convex set contained in the
closed Euclidean ball in $\R^{d+1}$ with radius $\sqrt{d/4+M^2}$ and center at $(1/2,1/2, \ldots , 1/2, 0)$.

On the other hand, for any $x\in \Omega$, $y \in \Omega$, and $f,g \in \cF_{\alpha}(\Omega)$,
\begin{align*}
| f(x) - g(x) |
& \le  | f(x) - f(y) | + | f(y) - g(x)| \\
& \le  \alpha \| x-y \|_2 + | f(y) - g(x) | \\
& \le  \sqrt{1+\alpha^2} \| (x,g(x)) - (y,f(y)) \|_2 .
\end{align*}
Taking the infimum on $y \in \Omega$ followed by the supremum on $x \in \Omega$, we find that
\begin{align*}
\| f - g \|_{\infty} \le \sqrt{1+\alpha^2} h ( \mbox{epi}(f), \mbox{epi}(g) ).
\end{align*}
Thus, by Lemma~\ref{Bronshtein}
\begin{align*}
\log N( \eta , \cC_\infty(\Omega) , \| \cdot \|_{\infty} )
& \le  \log N( (1+\alpha^2)^{-1/2} \eta ,  \cK(\sqrt{M^2+d/4}), h) \\
& \le  C_0 \{ \sqrt{(1+\alpha^2)(M^2+d/4)} \eta^{-1} \}^{d/2} .
\end{align*}
Thus there exist
 $N \le \exp ( C_0 \{ \sqrt{(1+\alpha^2)(M^2+d/4)} \eta^{-1} \}^{d/2})$ functions
$f_1, \ldots , f_N $ defined on $\Omega$, such that for each $f \in \cK(\Omega)$,
there exists some $f_i$, $i \in \{ 1, 2,\ldots , N \}$,
such that $| f(x) - f_i (x) | \le \eta$ for all $x \in \Omega$.  For each $i \in \{1, \ldots , N\}$ define
\begin{align*}
& \overline{f}_i (x) = \sup \{ f(x) : \ |f(x) - f_i (x) | \le \eta, \ f \in \cF_{\alpha} (\Omega) \}; \\
& \underline{f}_i (x) = \inf \{ f(x): \ |f(x) - f_i (x) | \le \eta, \ f \in \cF_{\alpha}(\Omega) \}
\end{align*}
for each $x \in \Omega$. Then we have
\begin{align*}
\| \overline{f}_i - \underline{f}_i \|_{\infty} \le 2 \eta .
\end{align*}
In particular this implies that for all $1 \le p < \infty$
\begin{align*}
\int_\Omega | \overline{f}_i (x) - \underline{f}_i (x) |^p d\lambda(x) \le (2\eta)^{p} .
\end{align*}
Letting $\eps = 2\eta$ we find that
\begin{align*}
\log N_{[\, ]} ( \eps , \cC_\infty(\Omega), \| \cdot \|_p)
& \le   \log N( 2\eta , \cC_\infty(\Omega) , \| \cdot \|_{\infty} )  \\
& \le     C_0 \{ \sqrt{(1+\alpha^2)(M^2+d/4)} (2 \eta)^{-1} \}^{d/2} \\
& =  2^{-d} C_0 \{ (1+\alpha^2) (4M^2+d)\} ^{d/4}  \eps^{-d/2} .
\end{align*}
\end{proof}

\subsection{Paring the Boundary}

In this subsection, we show that if we pare off the boundary of $\Omega$ by $\delta$,  consider the set
\begin{eqnarray}
\Omega_\delta=\set{x\in \Omega}{{\rm dist}(x,\partial \Omega)\ge \delta},
\label{DeltaErosionOfOmega}
\end{eqnarray}
and then consider functions restricted to $\Omega_\delta$, the entropy can be
estimated using Lemma~\ref{uniform-Lipschitz}.
The details are proved in the following two lemmas.

\begin{lemma}\label{bounded}
Let $\Omega$ be a
compact
convex set in $[0,1]^d$ with $| \Omega | \ge 1/d!$.
Then there exists a constant $\Lambda$ depending only on $d$,
such that for any $1\le r\le \infty$ and any $0<\delta\le 1$,
$$
\cC_r(\Omega)\subset \Lambda \delta^{-d/r}\cdot \cC_\infty(\Omega_\delta),
$$
where $\Omega_{\delta}$  
is as defined in (\ref{DeltaErosionOfOmega}).
In fact $\Lambda = \max \{ (d \Gamma (d/2) /\pi^{d/2} )^{1/r} , (d!) d 2^{d+2} \}$ works.
\end{lemma}
\begin{proof}
First, we show that if $f\in \cC_r(\Omega)$, then on $\Omega_\delta$
\begin{align}
f\ge -(d!)^{1/r} 2^{d+2}d.\label{lower-bound}
\end{align}
Let $x_0$ be a minimizer of $f$ on $\Omega_\delta$. If $f(x_0)\ge 0$, then there is
nothing to prove; otherwise, the set $K:=\set{x\in \Omega}{f(x)\le 0}$
is a closed convex set with $x_0$ as an interior point.
Denote $K_0=K-x_0$, and define
$$
K_\eta= \set{ x \in \Omega}{ x = x_0 + y, \ y \in (1+\eta)K_0 \setminus (1-\eta)K_0 },
$$
where $0<\eta<1$. We show that if $x\notin K_\eta$,
then $|f(x)|>\eta|f(x_0)|$.    Indeed, consider a function
$g$ on $\Omega$ defined so that: $g(x_0)=f(x_0)$,  $g(\gamma)=f(\gamma)$
for all $\gamma\in \partial K$, and $g$ is linear on the line segment
$$
L_\gamma:=\set{x\in \Omega}{x=x_0+t(\gamma-x_0), t\ge 0}.
$$
Then, by the convexity of $f$ on each $L_\gamma$, we have
$|f(x)|\ge |g(x)|$ on $\Omega$. Because for all $x\notin K_\eta$, $\|x-\gamma\|\ge \eta\|x_0-\gamma\|$, we have
$$|g(x)|=|g(\gamma)|+\frac{\|x-\gamma\|}{\|x_0-\gamma\|}|f(x_0)|>\eta|f(x_0)|.$$
Hence, on $\Omega\setminus K_\eta$, $|f(x)|\ge \eta|f(x_0)|$.

Because the volume of $K_\eta$ is bounded by
$[(1+\eta)^d-(1-\eta)^d]\cdot |K|\le d2^d\eta | \Omega | $,  we have
\begin{align*}
1\ge \int_{\Omega\setminus K_\eta} |f(x)|^rd\lambda(x)\ge (\eta|f(x_0)|)^r\cdot[1-d2^d\eta]\cdot | \Omega |.
\end{align*}
This implies that
$$
|f(x_0)|\le \eta^{-1}| \Omega |^{-1/r} (1-d2^d\eta)^{-1/r}.
$$
By choosing $\eta=[d2^d(1+1/r)]^{-1}$, we obtain
$$
|f(x_0)|\le | \Omega |^{-1/r} d2^d\left(1+\frac1r\right)(1+r)^{1/r}\le (d!)^{1/r}2^{d+2}d \le (d!) 2^{d+2} d .
$$
This proves (\ref{lower-bound}).

Next, we show that there exists a constant $\Lambda$ depending on $\Omega$ such that on $\Omega_\delta$,
$f(x)\le \Lambda \delta^{-d/r}$.

Let $z_0$ be a maximizer of $f$ on $\Omega_\delta$. If $f(z_0)\le 0$, there is
nothing to prove. So, we assume $f(z_0)>0$. Let $V=\set{x\in \Omega}{f(x)<f(z_0)}$.
Then $V$ is a convex set with $z_0$ at its boundary. There exists a 
hyperplane that separates $V$ and $z_0$.
This 
hyperplane separates $\Omega$ into two parts. On the part not containing $V$,
$f\ge f(z_0)$. In particular, $f\ge f(z_0)$ on the half of the ball centered at $z_0$
with radius $\delta$. Calling this half of the ball $W$, we have
$$1\ge \int_W|f(x)|^rd\lambda(x)\ge \frac{\pi^{d/2}}{d\Gamma(d/2)}\delta^df(z_0)^r,$$
which implies that
$$f(z_0)\le \left(\frac{d\Gamma(d/2)}{\pi^{d/2}}\right)^{1/r}\delta^{-d/r}.$$
Together with (\ref{lower-bound}) we obtain that there exists some
$\Lambda$ depending only on $d$ such that for all $x\in \Omega_\delta$,
$|f(x)|\le \Lambda \delta^{-d/r}$.
\end{proof}

\begin{lemma}\label{D0}
Let $\Omega$ be a closed convex set in $[0,1]^d$.
For any $1\le r\le \infty$ and any $0<\delta\le 1$,
$$
N_{[\,]}(\eps, \cC_r(\Omega),  \| \cdot \|_{L^p (\Omega_\delta)} )\le \expP{C_2\delta^{-d/2-d^2/r}\eps^{-d/2}},
$$
where $\Omega_{\delta}$
is as defined in (\ref{DeltaErosionOfOmega})
and $C_2$ is a constant depending only on $d$, $C_0$ in
Lemma~\ref{Bronshtein}, and $\Lambda$ in Lemma~\ref{bounded}.
\end{lemma}
\begin{proof}
We show that when restricted to $\Omega_\delta$, $f$ has a
Lipschitz constant bounded by $2^{2+d/r}\Lambda\delta^{-1-d/r}$. Indeed, by
Lemma~\ref{bounded}, $f$ is bounded by $2^{d/r}\Lambda\delta^{-d/r}$ on
$\Omega_{\delta/2}$. 
Note that $\Omega_\delta\subset \Omega_{\delta/2}\subset \Omega$. Thus
by
\cite{MR1385671},
problem 2.7.4 page 165,
$f$ is Lipschitz on $\Omega_{\delta}$ with Lipschitz constant
$2 (\delta/2)^{-1} 2^{d/r} \Lambda \delta^{-d/r} = 2^{2+d/r} \Lambda \delta^{-1-d/r}$.

Thus by Lemma~\ref{uniform-Lipschitz} it follows that
\begin{align}
\log N_{[\, ]}& (\epsilon, \cC_r(\Omega) , \| \cdot \|_{L^p (\Omega_\delta)} ) \nonumber \\
& \le 2^{-d}C_0 (1+2^{4+2d/r}\Lambda^2 \delta^{-2-2d/r})^{d/4}
          (4 \Lambda^2 \delta^{-2d/r} + d)^{d/4} \epsilon^{-d/2} \nonumber \\
&\le  C_2 \delta^{-d/2-d^2/r}\eps^{-d/2}
\label{LipschitzPieceControl}
\end{align}
for some constant $C_2$ depending only on $d$ and $r$.
\end{proof}

\subsection{Combining}
In this subsection, we prove a lemma that enables us to study metric entropy by decomposing the set $\Omega$.
\begin{lemma}[Union] \label{union}
If $\Omega = \cup_{i=1}^k \Omega_i$, then for all $1\le p<r\le \infty$,
\begin{align}
&N (\eps , \cC_r(\Omega), \| \cdot \|_{L^p (\Omega)} )
\le \prod_{i=1}^k N\left ( \delta_i, \cC_r(\Omega_i), \| \cdot \|_{L^p (\Omega_i)} \right ) ,\label{entropy-union}\\
& N_{[\, ]} (\eps , \cC_\infty(\Omega), \| \cdot \|_{L^p (\Omega)} )
\le \prod_{i=1}^k N_{[\, ]} \left ( \delta_i, \cC_\infty(\Omega_i), \| \cdot \|_{L^p (\Omega_i)} \right )
\label{bracket-union},
\end{align}
where $\eps=(\sum_{i=1}^k\delta_i^p)^{1/p}$.
Furthermore, if $\Omega_1, \Omega_2, \ldots, \Omega_k$ have disjoint interiors, then
\begin{align}
N (\eps , \cC_r(\Omega), \| \cdot \|_{L^p (\Omega)} )
\le 4^k \prod_{i=1}^kN\left ( \eta_i, \cC_r(\Omega_i), \| \cdot \|_{L^p (\Omega_i)} \right )\label{disjoint-union}
\end{align}
where
$$
\left(\sum_{i=1}^k\eta_i^{\frac{rp}{r-p}}\right)^{\frac{r-p}{rp}}\le 2^{-1/r}\eps,
$$
which is stronger than (\ref{entropy-union}) when $r<\infty$.
\end{lemma}

\begin{proof}
For each $i \in \{ 1, 2, \ldots , k\}$, there exists a set $\cN_i$ of $N_i$ elements, where
$$
N_i := N(\delta_i , \cC_r(\Omega_i), \| \cdot \|_{L^p (\Omega_i)} )
$$
such that, for each $f \in \cC_r(\Omega)\subset \cC_r(\Omega_i)$,
there exists $f_i \in \cN_i$
satisfying
\begin{eqnarray*}
\int_{\Omega_i} | f_i(x) - f(x) |^p d\lambda(x) \le  \delta_i^p .
\end{eqnarray*}
Define $ \hat{f}(x) = f_{i} (x)$ for $x \in \Omega_i \setminus \cup_{j < i} \Omega_{j}$, $1 \le i \le k$.
Then we have
$$
\int_\Omega | f (x) - \hat{f}(x) |^p d\lambda(x) \le \sum_{i=1}^k \int_{\Omega_i} | f(x) - f_i (x) |^p d\lambda(x)
\le \sum_{i=1}^k\delta_i^p = \eps^p .
$$
Since $\hat{f}$ is determined by $f_1, f_2, \ldots, f_k$, and each $f_i$ has at most $N_i$ possibilities,
the total number of possibilities for $\hat{f}$ is no more than $N_1N_2\cdots N_k$.
Thus, (\ref{entropy-union}) follows.

The proof of (\ref{bracket-union}) is similar. For each $i \in \{ 1, 2, \ldots , k\}$, there exists a set $\cN_i$ of $\hat{N}_i$ brackets, where
$$
\hat{N}_i := N_{[\,]} (\delta_i , \cC_\infty(\Omega_i), \| \cdot \|_{L^p (\Omega_i)} )
$$
such that, for each $f \in \cC_\infty(\Omega)\subset \cC_\infty(\Omega_i)$,
there exists a bracket $[\underline{f}_i, \overline{f}_i ] \in \cN_i$
satisfying $\underline{f}_j (x) \le f(x) \le \overline{f}_j(x) $ for all $x \in \Omega_i$, and
\begin{eqnarray*}
\int_{\Omega_i} | \overline{f}_i(x) - \underline{f}_i (x) |^p d\lambda(x) \le  \delta_i^p .
\end{eqnarray*}
Define $ \overline{f}(x) = \overline{f}_{i} (x)$,
$\underline{f}(x) = \underline{f}_i (x)$, $x \in \Omega_i \setminus \cup_{j < i} \Omega_{j}$, $1 \le i \le k$.
Then we have $\underline{f}(x) \le f(x) \le \overline{f}(x)$ for all $x \in \Omega$, and
$$
\int_\Omega | \overline{f} (x) - \underline{f} (x) |^p d\lambda(x) \le \sum_{i=1}^k \int_{\Omega_i} | \overline{f}_i (x) - \underline{f}_i (x) |^p d\lambda(x)
\le \sum_{i=1}^k\delta_i^p = \eps^p .
$$
That is,  $[\underline{f}, \overline{f}]$ is an $\eps$-bracket in $L^p (\Omega)$ which contains $f$.
Because there are no more than $\hat{N}_1 \hat{N}_2 \cdots \hat{N}_k$
possibilities for 
$[\underline{f}, \overline{f}]$,
(\ref{bracket-union}) follows.

Now we turn to the proof of (\ref{disjoint-union}).
For any $f\in \cC_r(\Omega)$, and for each $i=1,2,\ldots, k$,
define $n_i(f)$ as the smallest positive integer such that
$$
n_i(f)\ge k\int_{\Omega_i}|f(x)|^rd\lambda(x).
$$
Then, $n_i(f)< k\int_{\Omega_i}|f(x)|^rd\lambda(x)+1$, and using the fact that
$\sum_{i=1}^k\int_{\Omega_i}|f|^rd\lambda(x)\le 1$, we get
 $$
 n_1(f)+n_2(f)+\cdots+n_k(f)\le \sum_{i=1}^k\left( k\int_{\Omega_i}|f(x)|^rd\lambda(x)+1\right)\le 2k.
 $$
 Let
$$
\cI=\set{(n_1,n_2,\ldots, n_k)\in \N^{k}}{n_1+n_2+\cdots+n_k\le 2k}.
$$
For each $I=(i_1,i_2,\ldots, i_k)\in \cI$, define
$$
\cF_I=\set{f\in \cC_r(\Omega)}{n_j(f)=i_j, 1\le j\le k}.
$$
Then we have
$\cC_r(\Omega)=\cup_{I\in \cI}\cF_I$. Thus,
$$
N(\eps, \cC_r(\Omega),\|\cdot\|_{L^p(\Omega)}) \le \sum_{I\in \cI}N(\eps, \cF_I,\|\cdot\|_{L^p(\Omega)}).
$$
Note that for each $j=1,2,\ldots, k$, $\cF_I\subset (i_j/k)^{1/r}\cC_r(\Omega_j)$. Thus,
\begin{align*}
N(\eta_j,\cC_r(\Omega_j),\|\cdot\|_{L^p(\Omega_j)})
&=N((i_j/k)^{1/r}\eta_j,(i_j/k)^{1/r}\cC_r(\Omega_i),\|\cdot\|_{L^p(\Omega_j)})\\
&\ge N((i_j/k)^{1/r}\eta_j,\cF_I,\|\cdot\|_{L^p(\Omega_j)}).
\end{align*}
Therefore, for each $1\le j\le k$, there exists a set $\cN_j$ of
$Z_j:=N(\eta_j,\cC_r(\Omega_i),\|\cdot\|_{L^p(\Omega_j)})$ elements such that for each $f\in \cF_I$,
there exists $f_j\in \cN_j$ satisfying
$$
\int_{\Omega_j}|f(x)-f_j(x)|^pd\lambda(x)\le (i_j/k)^{p/r}\eta_j^p.
$$
If we define $\hat{f}(x)=f_j(x)$ for $x\in \Omega_j\setminus\cup_{r<j}\Omega_r$, then we have
\begin{align*}
\int_{\Omega}|f(x)-\hat{f}(x)|^p d\lambda(x)
&\le\sum_{j=1}^k (i_j/k)^{p/r}\eta_j^p\\
&\le \left(\sum_{j=1}^k\frac{i_j}{k}\right)^{\frac{p}{r}}\left(\sum_{j=1}^k \eta_j^{\frac{rp}{r-p}}\right)^{1-\frac{p}r}\le 2^{\frac{p}r}\left(\sum_{j=1}^k \eta_j^{\frac{rp}{r-p}}\right)^{1-\frac{p}r}\le\eps^p.
\end{align*}
Since, there are no more than $Z_1Z_2\cdots Z_k$ possibilities for 
$\hat{f}$, we obtain
$$
N(\eps, \cF_I,\|\cdot\|_{L^p(\Omega)})
\le \prod_{i=1}^k N\left ( \eta_i, \cC_r(\Omega_i), \| \cdot \|_{L^p (\Omega_i)} \right ).
$$
Note that $\cI$ has 
cardinality ${{2k}\choose{k}}<4^k$, and hence (\ref{disjoint-union}) follows.
\end{proof}

\subsection{With Finitely Many Facets}
In this subsection, we derive a metric entropy upper bound for ${\cC}_r (\Omega)$ 
when $\Omega$ has finitely many facets.
The idea is as follows: First, we pare off the boundary of $\Omega$ to get a smaller set on which
the functions are uniform Lipschitz, and can be taken care of by Lemma~\ref{D0}.
Next, we decompose the pared-off part into several smaller polytopes with a bounded number of
facets, and handle each of the smaller polytopes by scaling and paring of the boundary, and so on.
The final estimate is obtained by using Lemma~\ref{D0} and Lemma~\ref{union}, followed by iteration.
Readers who have little interest in the specific dependence of constants on the number of facets
can assume that $\Omega$ is a $d$-simplex with $d+1$ facets. For general polytopes, one can
triangulate it into $d$-simplices and apply Lemma~\ref{union} to get the estimate in
the next subsection. The only loss is that the constant obtained that way may be bigger than the one
derived directly from the the number of facets in some cases.

We first prove the upper bound with constant $Ck^\gamma$ with
some $\gamma>1$ for a closed convex polytope with $k$ facets.
We will use it later only for the case $k=d+1$.
However, since the proof is the same, we prove it for the general $k$.

By scaling, we can assume that $\Omega$ is contained in unit $d$-cube with volume at least $1/d!$.
Thus, there exists a point $O\in \Omega$ such that the distance between $O$
and the boundary of $\Omega$ is at least $\delta_0:=1/(2d d!)$.
This is because the boundary of $[0,1]^d$ has $(d-1)$-dimensional area
$2d$, and its projection onto $\Omega$ is a
contraction, thus, the boundary of $\Omega$  has $(d-1)$-dimensional area at most $2d$,
and by a Bonnesen-style inequality (Corollary 2, page 25 of \cite{MR519520}) the inradius of
$\Omega$ is at least its volume divided by the $(d-1)$-dimensional surface area of its boundary,
i.e. the inradius is at least $\delta_0$.

By otherwise using a translation, we can assume that $O$ is the origin.
Let $F_i$ be the $i$-th facet of $\Omega$ for $i = 1, \ldots , k$.
Let $V_i$ denote the convex hull of $F_i$ and $O$.
Then, $V_i$, $i \in \{ 1, \ldots , k \}$, form a partition of $\Omega$.
For $\delta<\delta_0 := \frac1{2d^2d!}$, let $D_0:=(1- \delta/\delta_0)\Omega$.
Define $\Omega_{i}=V_i\setminus  D_0^\circ$, where $D_0^\circ$ denotes the interior of $D_0$.
Then we have
$$
\Omega=D_0\cup \Omega_1\cup\Omega_2\cup \cdots \cup\Omega_k.
$$
Note that each $\Omega_i$ has no more than $k+1$ facets.
To see this, we first observe $V_i$ has at most $k$ facets.
Indeed, each of the facets of $V_i$ besides $F_i$ is the
convex hull of a $(d-2)$-dimensional face of $F_i$ and $O$.
However, each $(d-2)$-dimensional face of $F_i$ corresponds to the intersection of
$F_i$ and another facet of $\Omega$.
Thus, the number of $(d-2)$-dimensional faces of $F_i$ is at most $k-1$.
Therefore, the number of facets of $V_i$ is at most $k$.
Notice that $\Omega_i$ has one more facet than $V_i$.
Hence, the number of facets of $\Omega_i$ is at most $k+1$.
By (\ref{disjoint-union}) we have
\begin{align}
 N(\eps, \cC_r(\Omega), \cdot\|_{L^p(\Omega)})
 \le 4^{k+1}
 N(\eta_0, \cC_r(D_0), \| \cdot\|_{L^p(D_0)})
 \prod_{i=1}^{k}  N(\eta_i, \cC_r(\Omega_i), \| \cdot\|_{L^p(\Omega_i)}),\label{faces}
 \end{align}
where $\eta_0=2^{-\frac1p}\eps$, and
$$
\eta_i=2^{-\frac1p}\left(\frac{|\Omega_i|}{\sum_{i=1}^k |\Omega_i| }\right)^{\frac1p-\frac1r} \epsilon .
$$
Because $D_0\subset \Omega_\delta$, by Lemma~\ref{D0}, we have
$$
\log N(\eps, \cC_r(\Omega), \| \cdot\|_{L^p(D_0)})\le C_2\delta^{-\frac{d}{2}-\frac{d^2}{r}}(\eps)^{-d/2}.
$$
On the other hand, if we let $T_i$ be an affine transform that maps
$\Omega_i$ into $[0,1]^d$ so that the volume of $T_i(\Omega_i)$ is at least $1/d!$,
then by scaling (\ref{rescaling}), and using the fact that
$$
\sum_{i=1}^k|\Omega_i|=|\Omega\setminus(1-\delta/\delta_0)\Omega|
=[1-(1-\delta/\delta_0)^d]|\Omega|\le d \delta/\delta_0 ,
$$
 we have for each $1\le i\le k$,
\begin{align*}
N(\eta_i, \cC_r(\Omega_i), \| \cdot\|_{L^p(\Omega_i)})
\le N(K\eps, \cC_r(T_i(\Omega_i)), \| \cdot\|_{L^p(T_i(\Omega_i))}),
\end{align*}
where
$$K=2^{-1/p}[2d^2(d!)^2\delta]^{\frac1r-\frac1p}.$$
Plugging into (\ref{faces}), we obtain
\begin{align}
 \log N(\eps, \cC_r(\Omega), \| \cdot\|_{L^p(\Omega)})
 \le & \ (k+1)\log 4+C_2\delta^{-\frac{d}{2}-\frac{d^2}{r}}\eps^{-d/2}\nonumber\\
 & + \ \sum_{i=1}^k\log N(K\eps, \cC_r(T_i(\Omega_i)), \| \cdot\|_{L^p(T_i(\Omega_i))}).
\label{iteration1}
\end{align}
Now let $\cF_k$ consist of all closed convex sets in $[0,1]^d$ with at most $k$ faces and with
volume at least $1/d!$, 
and define
$$
g(k,\eps)=\sup\set{\log N(\eps, \cC_r(\Omega), \|\cdot\|_{L^p(\Omega)})}{\Omega\in \cF_k}.
$$
For notational simplicity, we denote
$M=C_2\delta^{-\frac{d}{2}-\frac{d^2}{r}}$.
Then (\ref{iteration1}) together with the fact that $(k+1)\log 4\le 4k-4$
(which follows from the fact that $k\ge 3$) implies
\begin{align}
g(k,\eps)+4 \le M\eps^{-d/2}+k[g(k+1, K\eps)+4], \label{iteration2}
\end{align}
which is equivalent to
 \begin{align*}
[g(k,\eps)+4]\eps^{d/2}\le M+\frac{k}{K^{d/2}}[g(k+1, K\eps)+4](K\eps)^{d/2}.
\end{align*}
Now, we choose $\delta$ so that $K^{d/2}=2k$. Then
$$M=C_2\delta^{-\frac{d}{2}-\frac{d^2}{r}}=C_3k^{\frac{(r+2d)p}{r-p}}.$$ Thus,
 \begin{align*}
[g(k,\eps)+4]\eps^{d/2}\le C_3k^{\frac{(r+2d)p}{r-p}}+\frac12[g(k+1, (2k)^{2/d}\eps)+4]((2k)^{2/d}\eps)^{d/2}.
\end{align*}
Hence, for any positive integer $m$, we have
$$[g(k,\eps)+4]\eps^{d/2}\le C_3\sum_{j=0}^{m-1}\frac{(k+j)^{\frac{(r+2d)p}{r-p}}}{2^{j}}+2^{-m}[g(k+m,L_m\eps)+4](L_m\eps)^{d/2},$$
where
$$L_m=\prod_{j=0}^{m-1}(2k+2j)^{2/d}.$$
In particular, if we choose $m$  to be the smallest integer so that $L_{m}\eps\ge 1$, then $g(k+m,L_{m}\eps)=0$, and we obtain
$$g(k,\eps)\le C_4 k^{\frac{(r+2d)p}{r-p}}\eps^{-d/2}.$$
This finishes the proof of the upper bound with constant
of the order $k^\gamma$ with $\gamma=\frac{(r+2d)p}{r-p}$.
\hfill $\Box$

\subsection{Upper Bound for Polytopes: Theorem 1 (ii)}
In this subsection, we obtain a metric entropy upper bound for the case when 
$\Omega$ is a convex polytope. Our method is to triangulate $\Omega$ 
into simplices and then use results in the last section and Lemma~\ref{union}.

Note that if $\Omega$ is a convex polytope with $v$ extreme points, then it has
no more than $2 v^{\floor{d/2}}$ facets; see \cite{MR1899299},  
Propositions 5.5.2 and 5.5.3, page 100.
Therefore, we immediately
obtain the upper bound with constant of the order $v^{\gamma\floor{d/2}}$. We show that this estimate can be improved to $v^{\ceil{d/2}}$.
Indeed, if $\Omega$ has $v$ vertices, then it is known that $\Omega$
can be triangulated into $m=O(v^{\ceil{d/2}})$ many $d$-simplices;
this is Corollary~2.3 of \cite{MR815583};  
see also
\cite{MR1620056}.
Thus, we can write
$\Omega=\cup_{i=1}^m D_i$, where $D_i$ are $d$-simplices.
Because each $D_i$ has only $(d+1)$-facets, by what we have proved above it follows that
$$
\log N(\eta_i, \cC_r(D_i),\|\cdot\|_{L^p(D_i)})\le C_5|D_i|^{\frac{d}{2p}-\frac{d}{2r}}\eta_i^{-d/2},
$$
where $C_5$ is a constant depending only on $p,r,d$. Now applying (\ref{disjoint-union}), with
$$
\eta_{i}=2^{-\frac{r-p}{p}}\left(\frac{|D_i|}{|\Omega|}\right)^{\frac1p-\frac1r}\eps,
$$
we immediately obtain
\begin{align*}
\log N(\eps, \cC_r(\Omega),\|\cdot\|_{L^p(\Omega)})&\le \sum_{i=1}^m
\log N(\eta_i, \cC_r(D_i),\|\cdot\|_{L^p(D_i)})\\
&\le C_6 m |\Omega|^{\frac{d}{2p}-\frac{d}{2r}}\eps^{-d/2}\le C_7v^{\ceil{d/2}}\eps^{-d/2}.
\end{align*}
This proves Part (ii) of Theorem~\ref{thm1}. The proof for the statement of the bracketing entropy when $r=\infty$ is similar, and the details are thus omitted.

\subsection{General Upper Bound: Theorem 2}
In this subsection we use simplicial approximation to 
establish an upper bound for the entropy of ${\cC}_r (\Omega)$ when $\Omega $ is a 
general compact convex set with non-empty interior. 

Fix $0<\eps<1$;  we choose smallest integer $s$
so that $2^{-s}|\Omega|\le [2^{-1/p}\eps]^{\frac{rp}{r-p}} | \Omega|$.
By the definition of $S_{\cD}(t,\Omega)$, $\Omega$ contains $m_1\le S_{\cD}(1/2,\Omega)$ $d$-simplices $D_{1,i}$, $1\le i\le m_1$,
so that the volume of $\Omega\setminus \cup_{i=1}^{m_1}D_{1,i}$ is at most
$2^{-1}|\Omega|$, and
the set $\Omega\setminus \cup_{i=1}^{m_1}D_{1,i}$ contains $m_2=S_{\cD} (1/4, \Omega)-m_1<S_{\cD}(1/4, \Omega)$
$d$-simplices $D_{2,j}, 1\le j\le m_2$, so that the volume of
$$
\Omega\setminus \cup_{i=1}^2\cup_{j=1}^{m_i}D_{i,j}
$$
is at most $2^{-2}|\Omega|$.
Continuing this way, we obtain a sequence of $d$-simplices
$D_{i,j}, 1\le j\le m_i, 1\le i\le s$ that are packed in $\Omega$ so that the
uncovered volume of $\Omega$ is at most $2^{-s}|\Omega|$. If we denote
$$
\widehat{\Omega}_i=\cup_{k=1}^i\cup_{j=1}^{m_k}D_{k,j},
$$
then for all $f\in \cC_r(\Omega)$,
$$
\int_{\Omega\setminus\widehat{\Omega}_{s}}|f|^p d\lambda
\le |\Omega\setminus\widehat{\Omega}_{s}|^{1-\frac{p}r}\le \frac{\eps^p}{2}|\Omega|^{1-\frac{p}{r}}.
$$
Hence,
\begin{align}
N(\eps|\Omega|^{\frac1p-\frac1r},\cC_r(\Omega),\|\cdot\|_{L^p(\Omega)})
\le N(2^{-1/p}\eps|\Omega|^{\frac1p-\frac1r},\cC_r(\widehat{\Omega}_{s}),\|\cdot\|_{L^p(\widehat{\Omega}_{s})}).\label{hat-Omega}
\end{align}
Next, we choose
$$
\eta_{i,j}=2^{-1/p}\left(\frac{|D_{i,j}|}{\sum_{j=1}^{m_i}|D_{i,j}|}\cdot
\frac{\alpha_i}{\sum_{k=1}^s\alpha_k} \right)^{\frac1p-\frac1r}\cdot\frac{\eps}{2}|\Omega|^{\frac1p-\frac1r},
$$
where
$$
\alpha_i :=(2^{-i}|\Omega|)^{1-\beta}[S_{\cD} (2^{-i},\Omega)]^{\beta}, \ \ \beta:= \frac{2pr}{2pr+(r-p)d}.
$$
Using the fact that
$\sum_{j=1}^{m_i}|D_{i,j}|\le 2^{-(i-1)} |\Omega|$, we have
$$\eta_{i,j}|D_{i,j}|^{\frac1r-\frac1p}\ge 2^{-1/p}
\left(\frac{\alpha_i}{2^{-(i-1)} |\Omega|\sum_{k=1}^s\alpha_k}\right)^{\frac1p-\frac1r}\cdot\frac{\eps}{2}|\Omega|^{\frac1p-\frac1r}.
$$
Thus, together with the fact that $m_i\le S_{\cD} (2^{-i},\Omega)$, we have
\begin{eqnarray*}
\lefteqn{\sum_{j=1}^{m_i}\log N(\eta_{i,j},\cC_r(D_{i,j}),\|\cdot\|_{L^p(D_{i,j})})} \\
&\le & S_{\cD} (2^{-i},\Omega)\cdot c\left[2^{-1/p}\left(\frac{\alpha_i}{2^{-i}|\Omega|
             \sum_{k=1}^s\alpha_k}\right)^{\frac1p-\frac1r}\cdot\frac{\eps}{2} |\Omega|^{\frac1p-\frac1r}\right]^{-d/2}\\
&= & c2^{\frac{d}2+\frac{d}{2p}}\left(\sum_{k=1}^{s}\alpha_k\right)^{\frac{(r-p)d}{2pr}}
          \alpha_i\cdot [\eps |\Omega|^{\frac1p-\frac1r}]^{-d/2}.
\end{eqnarray*}
Therefore, by (\ref{disjoint-union}) and (\ref{hat-Omega}) we have,
\begin{eqnarray*}
\lefteqn{\log N(\eps |\Omega|^{\frac1p-\frac1r}, \cC_r(\Omega), \|\cdot\|_{L^p(\Omega)})} \\
&\le & \log 4\sum_{i=1}^sS_{\cD} (2^{-i},\Omega)+c2^{\frac{d}2+\frac{d}{2p}}
          \left(\sum_{k=1}^{s}\alpha_k\right)^{1/\beta}\cdot [\eps |\Omega|^{\frac1p-\frac1r}]^{-d/2}.
\end{eqnarray*}

Let $\gamma := rp/(r-p)$.
Note that  $2^{-s} \le [2^{-1/p} \epsilon ]^{\gamma} \le 2^{-(s-1)}$,
and $S_{\cD}(t,\Omega) \ge S_{\cD}(2^{-i},\Omega)$ for $t \in [2^{-(i+1)},2^{-i}]$.  Thus it follows that
\begin{eqnarray*}
\sum_{i=1}^s \alpha_i
& = & \sum_{i=1}^s \left ( 2^{-i} | \Omega | \right )^{1-\beta} S_{\cD} (2^{-i} , \Omega )^{\beta} \\
& = & | \Omega |^{1-\beta} \sum_{i=1}^s 2^{-i} \left ( \frac{S_{\cD} (2^{-i}, \Omega )}{2^{-i}} \right )^{\beta}  \\
& = & 2 | \Omega |^{1-\beta} \sum_{i=1}^s \int_{2^{-i-1}}^{ 2^{-i}}  \left ( \frac{S_{\cD}(2^{-i}, \Omega )}{2^{-i}} \right )^{\beta}  dt \\
& \le  &  2 | \Omega |^{1-\beta} \sum_{i=1}^s
                 \int_{2^{-i-1}}^{ 2^{-i}}  \left ( \frac{S_{\cD}(t, \Omega )}{t} \right )^{\beta}  dt \\
& = &  2 | \Omega |^{1-\beta}
                 \int_{2^{-s-1}}^{2^{-1}}  \left ( \frac{S_{\cD}(t, \Omega )}{t} \right )^{\beta}  dt\\
 & \le & 2 | \Omega |^{1-\beta}
                 \int_{2^{-2} \cdot [2^{-1/p}\eps]^{\gamma}}^1  \left ( \frac{S_{\cD}(t, \Omega )}{t} \right )^{\beta}  dt  .
\end{eqnarray*}
Hence,
$$
c2^{\frac{d}2+\frac{d}{2p}}
          \left(\sum_{k=1}^{s}\alpha_k\right)^{1/\beta}\cdot [\eps |\Omega|^{\frac1p-\frac1r}]^{-d/2}
\le c2^{\frac{d}2+\frac{d}{2p}+\frac1\beta}
      \left(\int_{\delta(\eps)}^{1}\left(\frac{S_{\cD}(t,\Omega)}{t}\right)^{\beta}dt\right)^{1/\beta}\cdot\eps^{-d/2},
$$
where $\delta(\eps)=2^{-2} \cdot [2^{-1/p}\eps]^{\gamma}$.

Similarly,
$$
\sum_{i=1}^sS_{\cD} (2^{-i},\Omega)
\le  2 \int_{2^{-2} \cdot [2^{-1/p}\eps]^{\gamma}}^1 \frac{S_{\cD}(t,\Omega)}{t}dt.
$$
Hence, we obtain
\begin{eqnarray*}
\lefteqn{\log N(\eps|\Omega|^{\frac1p-\frac1r},\cC_r(\Omega), \|\cdot\|_{L^p(\Omega)})} \\
&\le & C\int_{\delta(\eps)}^{1}\frac{S_{\cD}(t,\Omega)}{t}dt
   + C\left(\int_{\delta(\eps)}^{1}\left(\frac{S_{\cD}(t,\Omega)}{t}\right)^{\beta}dt\right)^{1/\beta}\cdot\eps^{-d/2}
\end{eqnarray*}
with $C=\max\left\{2\log 4, c2^{\frac{d}2+\frac{d}{2p}+\frac1\beta}\right\}$.

\subsection{Worst Case Upper Bound: Theorem 3}
In this subsection, we will compute a concrete general metric entropy upper bound using Theorem~\ref{thm2}.

From Example~\ref{example} in the introduction, we know that when $\Omega$ is the closed unit ball
$B_d (0,1)$ in $\R^d$ with $d=2$, then there exists a simplicial approximation sequence $\cD=\{D_1,D_2,\ldots\}$ such that $S_{\cD}(t,\Omega)=O(t^{-\frac{d-1}{2}})$ for all $0<t<1$.

Now, we assume that $\Omega\subset B_d(0,1)$ is a general closed convex set with
non-empty interior, where $B_d(z,r)$ is the closed ball in $\R^d$ with radius $r$ and center at
$z$. Instead of constructing each simplex in the sequence individually,
we will construct a sequence of inscribed polytopes, and triangulate them into simplices.
To construct these polytopes, we do not work on $\Omega$ directly. Instead,
we first let $\widetilde{\Omega}=\Omega+B_d(0,1)$, and for any $0<t<1$, we use a known result
(See e.g. the proof of Lemma 8.4.14 of \cite{MR3445285}; or Lemma 10 of \cite{CG} when
$\Omega$ is only of positive reach. Actually, \cite{MR3445285} proved for $\Omega+B_d(0,2)$ instead
of $\widetilde{\Omega}=\Omega+B_d(0,1)$, but with a slightly different construction, the statement also holds.):

\noindent {\it There exists a simplicial sphere $\widetilde{P}$ (an inscribed convex
polytope in $\widetilde{\Omega}$ whose facets are $(d-1)$-simplices) with $O(t^{-(d-1)/2})$
facets such that $\widetilde{\Omega}\subset \widetilde{P}+B_d(0,t)$.}

\noindent Then, we use $\widetilde{P}$ to construct a simplicial sphere $P$ in
$\Omega$ with $O(t^{-(d-1)/2})$ facets such that $\widetilde{\Omega}\subset \widetilde{P}+B_d(0,t)$.
Finally, we triangulate a sequence of such simplicial spheres to construct a simplicial approximation
sequence $\cD$ for $\Omega$, such that
$S_{\cD}(t,\Omega)=O(t^{-\frac{d-1}{2}})$ for all $0<t<1$.

For the convenience of  readers who are interested in knowing how the
simplicial spheres are constructed, we provide a proof for the aforementioned known result:
Since $\Omega\subset B_d(0,1)$, we have $\widetilde{\Omega} \subset [-2,2]^d$. For any integer $n>1$, we divide
each facet of $[-2,2]^d$ into $(4n)^{d-1}$ closed $(d-1)$-cubes
of side-length $1/n$.
Each of these small $(d-1)$-cubes can be triangulated into no more
than $d!$ closed $(d-1)$-simplices. Thus, the boundary of $[-2,2]^d$ can be triangulated into
$m_k\le (4n)^{d-1}dd!$ closed $(d-1)$-simplices, each of which has diameter
most $\sqrt{d}/n$. Let $K_i$, $1\le i\le m_k$ be these simplices. Clearly, the set of all
vertices of these simplices forms a $\sqrt{d}/n$-net of the boundary of $[-2,2]^d$.
Each $K_i$ has $d$ vertices. The projections of these vertices onto $\widetilde{\Omega}$
form a $(d-1)$-simplex with vertices on the boundary of $\widetilde{\Omega}$.
Denote this $(d-1)$-simplex by $\widetilde{\Delta}_i$. Because a projection onto a
convex set is a contraction, the diameter of $\widetilde{\Delta}_i$ is no larger than $\sqrt{d}/n$,
and the set of all vertices of these simplices forms a $\sqrt{d}/n$-net of the boundary of $\widetilde{\Omega}$.
Let $\widetilde{P}$ be the convex hull of $\widetilde{\Delta}_i, 1\le i\le m_k$. Then $\widetilde{P}$ is a simplicial
sphere contained in $\widetilde{\Omega}$ with $m_k\le (4n)^{d-1}dd!$ facets, each of which has a diameter
no larger than $\sqrt{d}/n$. Furthermore, $\widetilde{\Omega}\subset \widetilde{P}+B_d(0,\sqrt{d}/n)$,
which implies that $\Omega$ is contained in the interior of $\widetilde{P}$ if $n>\sqrt{d}$.
We show that for $n>\sqrt{d}$, we actually have $\widetilde{\Omega}\subset \widetilde{P}+B_d(0,d/n^2)$.
Indeed, for any $x$ on the boundary of $\widetilde{\Omega}$, by the definition of $\widetilde{\Omega}$,
there exists $y\in \Omega$, such that ${\rm dist}(y,\partial \widetilde{\Omega})={\rm dist}(x,y)=1$.
Because $y$ is an interior point of $\widetilde{P}$. The line segment $yx$ intersects the boundary of
$\widetilde{P}$ at some point say $u$. Let $F$ be a facet of $\widetilde{P}$ that contains $u$.
Consider the hyperplane passing through $u$ and orthogonal to $xy$.
Because all the vertices of $F$ are outside the interior of the ball $B_d(y,1)$, if ${\rm dist}(y,u)<\sqrt{1-d/n^2}$,
then the vertices and $y$ must lie on different sides of the hyperplane $H$.
Since $u$ is a convex combination of these vertices, $u$ cannot lie on the hyperplane $H$, which is a contradiction.
Thus,
$$
{\rm dist}(x,u)={\rm dist}(x,y)-{\rm dist}(u,y)\le 1-\sqrt{1-d/n^2}<d/n^2.
$$
Therefore, ${\rm dist}(x,\widetilde{P})<d/n^2$. Consequently,
$\widetilde{\Omega}\subset \widetilde{P}+B_d (0,d/n^2)$.
In particular, for any $0<t<1$, we choose $n$ as the smallest positive integer such that $d/n^2<t$.
Then the simplicial sphere $\widetilde{P}$ has $O(t^{-(d-1)/2})$ facets, and the claim follows.\\

Now we construct a simplicial sphere $P$ on $\Omega$ with $O(t^{-{(d-1)}/2})$ facets
such that $\Omega\subset P+B(0,t)$. For each facet $\widetilde{\Delta}_i$ of  $\widetilde{P}$,
we project its $d$ vertices onto $\Omega$. The projections of these vertices onto $\Omega$ form a
$(d-1)$-simplex. (If it is degenerate,  we simply do not include it in our next step). Denote it by $D_i$.
Let $P$ be the convex hull of these $D_i$. Thus $P$ is a simplicial sphere in $\Omega$ with
$O(t^{-{(d-1)}/2})$ facets. It remains to show that $\Omega\subset P+B_d(0,t)$.
Indeed, for any $U$ on the boundary of $\Omega$, let $V$ be the projection of $U$ onto $P$.
Suppose ${\rm dist}(U,V)>t$. The ray starting from $V$ and containing $U$ intersects the boundary of
$\widetilde{\Omega}$ at some point $W$. Since $P$ is convex,
${\rm dist}(W, P)={\rm dist}(W,U)+{\rm dist}(U,V)>1+t$. Thus $W\notin P+B_d(0,1+t)$. However,
$$P+B_d(0,1+t)=P+B_d(0,1)+B_d(0,t)\supset \widetilde{P}+B_d(0,t)\supset\widetilde{\Omega}\ni W.$$
This is a contradiction.
Hence ${\rm dist}(U,V)\le t$. This implies that for any $U\in \partial \Omega$,
${\rm dist}(U,P)\le t$. Therefore, $\Omega\subset P+B_d(0,t)$.

From what we have proved so far, we can summarize that for any $d/n^2<1$,
there exists a simplicial sphere $P_n$ that contains no more than $(4n)^{d-1}dd!$
facets of diameter at most $\sqrt{d}/n$ such that $\Omega\subset P+B_d(0,d/n^2)$.
Furthermore, each facet of $P_n$ is generated through a $(d-1)$-simplex that is
contained in a $(d-1)$-cube of edge-length $1/n$ on the boundary of $[-2,2]^d$.\\

Now, we construct a simplicial approximation sequence of $\Omega$ as follows.
Let $k=\floor{\sqrt{d}}+1$. Then $d/k^2<1$. Let $O$ be a fixed interior point of $P_k$.
For each facet $F_i$ of $P_k$, let $D_i$ be the convex hull of $F_i$ and $O$.
Thus, $D_i$ is a $d$-simplex, and $P_k$ can be partitioned into $d$-simplices
$D_1, D_2, \ldots, D_{s_1}$, where $s_1\le (4k)^{d-1}dd!$ is the number of facets in $P_k$.

Consider the set $P_{2k}\setminus P_k$. For each facet $J_i$ of $P_{2k}$, let $Q_i$ be the
convex hull of $J_i$ and $O$. Because $J_i$ is generated by a $(d-1)$-simplex contained in a
$(d-1)$-cube of edge-length $6/(2k)$ on the boundary of $[-2,2]^d$, which only intersects with
no more than $d!$ $(d-1)$-simplices that generate the facets of $P_{k}$, the $d$-simplex 
$Q_i$ intersects with at most $d!$ facets of $P_k$. Thus each set 
$Q_i\cap(P_{2k}\setminus P_k)$ can be triangulated into at most $c(d)$ $d$-simplices, 
where $c(d)$ is a constant depending only on $d$. Consequently, $P_{2k}\setminus P_k$ 
can be triangulated into no more than $(8k)^{d-1}dd!\cdot c(d)$ simplices. 
Denote these simplices by $D_{s_1+1}, D_{s_1+2}, \ldots, D_{s_2}$.

We continue this process for $P_{3k}\setminus P_{2k}$ and so on to obtain a
simplicial approximation sequence $\cD=\{D_1, D_2, \ldots\}$.
Now we estimate $S_{\cD}(t,\Omega)$. For any $0<t<1$, we choose $r$ to be the
smallest integer such that $\sigma_{d-1}d2^{-2(r-1)}<t$, where $\sigma_{d-1}$ is the
surface of $d$-dimensional unit surface. Thus, for $n\ge s_r$, we have
$$|\Omega\setminus \cup_{i=1}^nD_i|\le |\Omega\setminus P_{2^{r-1}k}|\le \sigma_{d-1}\cdot d2^{-2(r-1)}<t.$$
Hence,
\begin{eqnarray*}
S_{\cD}(t,\Omega)
& \le & s_r 
      \le  (4k)^{d-1}dd!+ (8k)^{d-1}dd!\cdot c(d)+\cdots +(2^{r-1}\cdot 2k)^{d-1}dd!\cdot c(d)\\
& \le & k_d t^{-(d-1)/2},
\end{eqnarray*}
 where $k_d$ is a constant depending only on $d$.

A direct computation of the integrals in Theorem~\ref{thm2} gives the concrete upper bounds stated in Theorem~\ref{thm3}. (When $p<\frac{dr}{d+(d-1)r}$, the term
$\eps^{-d/2}$ comes from the second integral.)

The proof for the statement of the bracketing entropy when $r=\infty$ is similar, and the details are thus omitted.

\subsection{General Lower Bound: Theorem 1 (i)}
In this subsection, we prove the general lower bound stated in Theorem~\ref{thm1} (i).
By Lemmas~\ref{box} and~\ref{scaling}, we only need to prove it for the case when
$\Omega$ is contained in $[0,1]^d$ and has volume at least $1/d!$.
Indeed, by Lemma~ \ref{box}, if $\Omega \subset \R^d$ is closed
and convex, $\Omega \subset R$ for a box $R$ with
$|R| \le  d! | \Omega |$. Let $T$ be any affine transformation that maps $R$ onto $[0,1]^d$,
Then by Lemma~\ref{scaling}, and the fact that $C_r(T(\Omega))\supset C_\infty(T(\Omega))$, we have
\begin{align*}
N(|\Omega|^{\frac1p-\frac1r}\eps , \cC_{r} (\Omega), \| \cdot \|_{L^p (\Omega)})
&= N( \left(|\Omega|/|R|\right)^{\frac1p-\frac1r}\eps, \cC_r (T(\Omega) ), \| \cdot \|_{L^p (T(\Omega))})\\
&\ge  N(\eps, \cC_\infty (T(\Omega) ), \| \cdot \|_{L^p (T(\Omega))}).
\end{align*}
Thus it suffices to establish a lower bound for the case when
$\Omega$ is contained in $[0,1]^d$ and has volume at least $1/d!$.

We choose a function $f$ so that $f$ is supported on $[0,1]^d$, with $0\le f\le \frac1{20}$
and $\|f\|_1\ge \frac{1}{80d}$. Furthermore, the Hessian matrix of $f$ at every
$(x_1,x_2,\ldots,x_d)\in[0,1]^d$ is a diagonal matrix with each entry bounded by $1$.
One such function is
$$
f(x_1,x_2,\ldots,x_d)=\twopiece[1]{\frac1{20d}\sum_{i=1}^d
          \sin^3(\pi x_i)}{(x_1,x_2,\ldots,x_d)\in[0,1]^d}{0}{(x_1,x_2,\ldots,x_d)\notin[0,1]^d}.
$$
For each fixed $0<\eps<(10 d!)^{-2}$, and each $I=(i_1,i_2,\ldots,i_d)\in \N^d$, define
$$
f_I(x_1,x_2,\ldots, x_d)=\eps^2\cdot f\left(\frac{x_1-i_1\eps}{\eps},
            \frac{x_2-i_2\eps}{\eps},\ldots,\frac{x_d-i_d\eps}\eps\right).
$$
Then, $f_I$ is supported on
$$
B_I:=[i_1\eps,(i_1+1)\eps]\times [i_2\eps,(i_2+1)\eps]\times\cdots\times [i_d\eps,(i_d+1)\eps]
$$
with $0\le f_I\le \frac{\eps^2}{20}$, $\|f_I \|_1\ge \frac{\eps^2}{80d}\cdot \eps^{d}$,
and furthermore, the Hessian matrix of $f_I$ at every
$(x_1,x_2,\ldots,x_d)\in B_I$ is a diagonal matrix with each entry bounded by $1$.

Denote $$\cI=\set{I}{I=(i_1,i_2,\ldots,i_d)\in\N^d, B_I\subset \Omega}.$$
Let $\xi_I\in \{0,1\}$, $I\in \cI$ be i.i.d. random variables with $\P(\xi_I=1)=\P(\xi_I=0)=1/2$, and define
the random function
$$
F(x;\xi)=\sum_{I\in \cI} \xi_I f_I(x).
$$
Then for each realization of $\xi=(\xi_I)_{I\in \cI}$, we have $0\le F\le \frac{\eps^2}{20}$,
and the Hessian matrix of $F$ is diagonal with each entry bounded by $1$.
Therefore, for each realization of $\xi$ the function
$$
G(x;\xi)=\frac1d\left(x_1^2+x_2^2+\cdots+x_d^2-F(x;\xi)\right)
$$
is convex and bounded by $1$. Hence, $G(\cdot;\xi)\in \cC_\infty([0,1]^d)$.

There are $2^{|\cI|}$ realizations of $G(\cdot;\xi)$. Between two realizations, we define the Hamming distance
$$
H(G(\cdot;\xi^{(1)}), G(\cdot;\xi^{(2)}))=\#\set{I\in \cI}{\xi^{(1)}_I\ne\xi^{(2)}_I}.
$$
For $r=\floor{|\cI|/10}$, consider the set
$$
U(G(\cdot;\xi),r)=\left \{ G(\cdot;\xi^{(2)}) : \ H(G(\cdot;\xi), G(\cdot;\xi^{(2)}))\le r \right \} .
$$
For each $G(\cdot;\xi)$, the set $U(G(\cdot;\xi),r)$ contains no more than
$$
\sum_{k=0}^r{{|\cI|}\choose{k}}\le 2^{9|\cI|/10}
$$
elements. Thus, by the pigeonhole principle, we can find
$m\ge 2^{|\cI|}\div 2^{9|\cI|/10}=2^{|\cI|/10}$ realizations of
$G(\cdot;\xi^{(k)})$, $1\le k\le m$, such that for any $1\le i<j\le m$, we have
$$
H(G(\cdot;\xi^{(i)}), G(\cdot;\xi^{(j)}))\ge \floor{|\cI|/10}.
$$
Note that
\begin{align}
\int_{\Omega}\left|G(x;\xi^{(i)})-G(x;\xi^{(j)})\right| d\lambda(x)
    =& \ \frac1{d^2}\int_{\Omega}\sum_{I\in \cI}|\xi_I^{(i)}-\xi_I^{(j)}||f_I(x)|d\lambda(x)\nonumber\\
\ge & \ \frac1{d} \sum_{I\in \cI}|\xi_I^{(i)}-\xi_I^{(j)}|\frac{\eps^2}{80d}\cdot \eps^{d}\nonumber\\
\ge & \ \frac1{d}\cdot\floor{|\cI|/10}\cdot \frac{\eps^2}{80d}\cdot \eps^{d}.\label{L2distance}
\end{align}
We show that the cardinality $|\cI|$ of $\cI$ is at least $\frac1{2d!}\eps^{-d}$.
Indeed, because $\Omega\subset [0,1]^d$ has volume at least $1/d!$, and
$\Omega$ is convex, so the set $[0,1]^d\setminus \Omega_{\sqrt{d}\eps}$ has volume at most
$1-1/d!+2d\cdot \sqrt{d}\eps$. Thus, $[0,1]^d\setminus \Omega_{\sqrt{d}\eps}$ contains no more than
$\eps^{-d}\cdot [1-1/d!+2d\cdot \sqrt{d}\eps]$ cubes $B_I$. Any cube $B_I\subset [0,1]^d$
that is not contained in $[0,1]^d\setminus \Omega_{\sqrt{d}\eps}$ does not intersect with
$[0,1]^d \setminus \Omega$, thus must be contained in $\Omega$.
Since $[0,1]^d$ contains $\floor{1/\eps}^d$ such cubes, and we conclude that $\Omega$ contains at least
$$
\floor{1/\eps}^d-\eps^{-d}\cdot \left[1-1/d!+2d\cdot \sqrt{d}\eps\right]\ge \frac1{2d!}\eps^{-d}
$$
cubes provided that $\eps$ is small, say $\eps<(10d!)^{-2}$.

Now plugging the inequality $|\cI|\ge \frac{1}{2d!}\eps^{-d}$ into (\ref{L2distance}), we obtain
$$
\int_{\Omega}\left|G(x;\xi^{(i)})-G(x;\xi^{(j)})\right|d\lambda(x)\ge c\eps^2,
$$
for some constant $c$ depending only on $d$. This implies that $\cC_\infty([0,1]^d)$ contains
$$
m\ge 2^{|\cI|/10}\ge e^{c'\eps^{-d}}
$$
functions whose mutual $L^1(\Omega)$ distance is at least $c \eps^2$.
This implies that
$$
\log N(\eps, \cC_\infty([0,1]^d),\|\cdot\|_{L^1(\Omega)})\ge c''\eps^{-d/2}
$$
for some $c''>0$ depending only on $d$.

Since $|\Omega|\ge \frac1{d!}$, for any $p\ge 1$, we have
$\|\cdot\|_{L^p(\Omega)}\ge (d!)^{-\frac{p-1}{p}} \|\cdot\|_{L^1(\Omega)}$, this implies that
$$
\log N(\eps, \cC_\infty([0,1]^d),\|\cdot\|_{L^p(\Omega)})\ge c\eps^{-d/2}
$$
for some constant $c$ depending on $p$ and $d$, provided that $|\Omega|\ge \frac1{d!}$.

Together with the discussion at the beginning of this subsection,
and the fact that bracketing entropy is bounded below by metric entropy
we conclude that the lower bound statements of Theorem~\ref{thm1} are true.

\subsection{Lower Bound for the Ball: Theorem 4}
The $(d-1)$-dimensional area of the unit sphere $S^{d-1}$ in $\R^d$ is $2 \pi^{d/2} / \Gamma (d/2)$,
while the $(d-1)$-dimensional area of a cap with height $h$ is
$(\pi^{d/2}/ \Gamma (d/2)) I_{2h-h^2} ((d-1)/2, 1/2) \sim c_d h^{(d-1)/2}$ where
$I_x (a,b)$ is the regularized incomplete beta function.  Thus there exist
$s:= \alpha_d h^{- (d-1)/2}$ disjoint spherical caps with height $h$.
The $d$-dimensional volume of each spherical cap is $\beta_d h^{(d+1)/2}$.
Let $x_1, \ldots , x_s$ be the spherical center of the caps.
For each $1 \le i \le s$, we define a random function $f_i$ on the closed unit ball $\Omega = B_d (0,1)$ such that, for $y \in B_d (0,1)$,
\begin{eqnarray*}
f_i (y) = \left \{ \begin{array}{l l} 0, & \langle y,x_i \rangle \le 1-h, \\
                                                  \xi_i \frac{\langle y,x_i \rangle - (1-h)}{h}, & \langle y,x_i \rangle > 1-h ,
                       \end{array} \right .
\end{eqnarray*}
where $\xi_i$ is either $0$ or $1$.
Now $f_i$ is convex on the closed unit ball, and supported on the $i$-th cap $C_i$;  $f_i (y) = 0 $ if $y/\| y\|_2 \notin C_i$.
Furthermore, since the caps are disjoint, the sum $f = \sum_{i=1}^s f_i$ is also convex and bounded
by $1$.
There are $2^s$ different possibilities for $f$.
By the same argument as we
used in the proof of the lower bound of Theorem~\ref{thm1}, we can find a set
$W$ of $2^{s/2}$ functions in which any two functions $f$ and $g$ are different on at least $s/10$ caps.

On each cap where the two functions are defined differently, $|f - g| \ge 1/2$ the top half height of the cap which has
a volume $\gamma_d h^{(d+1)/2}$.
Consequently the $L^p $ distance between any two functions $f,g \in W$ is at least
\begin{eqnarray*}
\frac{1}{2} ( s/10 \cdot \gamma_d h^{(d+1)/2} )^{1/p} \ge \delta_d h^{1/p} .
\end{eqnarray*}
Letting $\delta_d h^{1/p} = \eps$ we have
\begin{align*}
N( \eps , \cC_\infty(B_d(0,1)), \| \cdot \|_{L^p (B_d (0,1)) } ) \ge \exp \left ( C \eps^{- (d-1)p/2 } \right ) .
\end{align*}
When $(d-1)p\le d$, the lower bound above should be replaced
by the universal lower bound $\eps^{-d/2}$ proved in the last section.

\subsection{Optimality of Lower Bound for Polytopes: Remark 1}
Clearly, it is enough to show (\ref{linear}) for the case $r=\infty$, and $p=1$. Let $\Omega$ be the regular regular $(m+2)$-gon inscribed in the unit circle, which can be triangulated into $m$ triangles.
If $m<4$, the statement simply follows from Theorem~\ref{thm1} (i) with $c_2=c_1/6$. If $m\ge 4$, by connecting every other vertices, we can cut off $n=\floor{m/2}$ isosceles triangles from $\Omega$. Denote these isosceles triangles by $\Delta_i$, $1\le i\le n$. Each $\Delta_i$ has base-length $2\sin(\frac{\pi}{m+2})$, and height $1-\cos(\frac{\pi}{m+2})$. If $\eps<\frac14 m^{-2}$, each $\Delta_i$ contains $cm^{-3}\eps^{-2}\ge 2$ disjoint squares of side-length $\eps$. All together, these $n$ isosceles triangles contain $c'm^{-2}\eps^{-2}$ disjoint squares of side-length $\eps$ for some constant $c'$. We denote by $\cJ$ the class of these small squares.

Note that the base line of each isosceles triangles separates the isosceles triangle from the rest of $\Omega$.
If $f_0$ is a fixed function, and each $f_i$ is a function defined on $\Omega$ and supported on $\Delta_i$,
such that $f_0\pm f_i$ is convex, then for all choices of $\eps_i\in \{0,1\}$, $1\le i\le n$, the function
$g=f_0+\sum_{i=1}^n \eps_if_i$
is also convex on $\Omega$. Therefore, replace the class $\cI$ in $\S 2.9$ by the class $\cJ$ defined above, the same argument in $\S 2.9$ gives
$$\log N(\eps, \Omega,\|\cdot\|_p)\ge c'' \sqrt{|\cJ|}\ge c''m\eps^{-d/2}$$
for some constants $c''$ and $c'''$. This finishes the proof of the statement in Remark 1.\\

\noindent {\bf Acknowledgement.}
The authors thank the anonymous referees for their valuable comments resulting
in improvements to both the results and the presentation of the paper.

\providecommand{\bysame}{\leavevmode\hbox to3em{\hrulefill}\thinspace}
\providecommand{\MR}{\relax\ifhmode\unskip\space\fi MR }
\providecommand{\MRhref}[2]{%
  \href{http://www.ams.org/mathscinet-getitem?mr=#1}{#2}
}
\providecommand{\href}[2]{#2}



\begin{thebibliography}{10}

\bibitem{MR722129}
Lucien Birg\'e, \emph{Approximation dans les espaces m\'etriques et th\'eorie
  de l'estimation}, Z. Wahrsch. Verw. Gebiete \textbf{65} (1983), no.~2,
  181--237. \MR{722129}

\bibitem{MR1240719}
Lucien Birg\'e and Pascal Massart, \emph{Rates of convergence for minimum
  contrast estimators}, Probab. Theory Related Fields \textbf{97} (1993),
  no.~1-2, 113--150. \MR{1240719}

\bibitem{MR0415155}
E.~M. Bronshtein~[Bron{\v{s}}tein], \emph{$\epsilon-$entropy of convex sets and
  functions}, Siberian Mathematical Journal \textbf{17} (1976), 393--398,
  Transl. from {\sl Sibirsk. Mat. \v{Z}.} {\bf 17}, 508--517.

\bibitem{CG}
James Cockreham and Fuchang Gao, \emph{Entropy estimate for classes of set with
  positive reach}, (preprint) (2016).

\bibitem{MR1620056}
T.~K. Dey and J.~Pach, \emph{Extremal problems for geometric hypergraphs},
  Discrete Comput. Geom. \textbf{19} (1998), no.~4, 473--484. \MR{1620056
  (99c:05144)}

\bibitem{MR2519658}
D.~Dryanov, \emph{Kolmogorov entropy for classes of convex functions}, Constr.
  Approx. \textbf{30} (2009), no.~1, 137--153. \MR{2519658 (2010h:41037)}

\bibitem{MR0358168}
R.~M. Dudley, \emph{Metric entropy of some classes of sets with differentiable
  boundaries}, J. Approximation Theory \textbf{10} (1974), 227--236.
  \MR{0358168}

\bibitem{MR876079}
\bysame, \emph{A course on empirical processes}, \'Ecole d'\'et\'e de
  probabilit\'es de {S}aint-{F}lour, {XII}---1982, Lecture Notes in Math., vol.
  1097, Springer, Berlin, 1984, pp.~1--142. \MR{876079}

\bibitem{MR1720712}
\bysame, \emph{Uniform {C}entral {L}imit {T}heorems}, Cambridge Studies in
  Advanced Mathematics, vol.~63, Cambridge University Press, Cambridge, 1999.
  \MR{1720712 (2000k:60050)}

\bibitem{MR3445285}
\bysame, \emph{Uniform {C}entral {L}imit {T}heorems}, second ed., Cambridge
  Studies in Advanced Mathematics, vol. 142, Cambridge University Press, New
  York, 2014. \MR{3445285}

\bibitem{MR2386068}
Fuchang Gao, \emph{Entropy estimate for {$k$}-monotone functions via small ball
  probability of integrated {B}rownian motion}, Electron. Commun. Probab.
  \textbf{13} (2008), 121--130. \MR{2386068 (2008m:60063)}

\bibitem{MR2520591}
Fuchang. Gao and Jon~A. Wellner, \emph{On the rate of convergence of the
  maximum likelihood estimator of a {$k$}-monotone density}, Sci. China Ser. A
  \textbf{52} (2009), no.~7, 1525--1538. \MR{2520591}

\bibitem{MR3474567}
Adityanand Guntuboyina, \emph{Covering numbers of {$L_p$}-balls of convex
  functions and sets}, Constr. Approx. \textbf{43} (2016), no.~1, 135--151.
  \MR{3474567}

\bibitem{MR3043776}
Adityanand Guntuboyina and Bodhisattva Sen, \emph{Covering numbers for convex
  functions}, IEEE Trans. Inform. Theory \textbf{59} (2013), no.~4, 1957--1965.
  \MR{3043776}

\bibitem{MR0124720}
A.~N. Kolmogorov and V.~M. Tihomirov, \emph{{$\varepsilon $}-entropy and
  {$\varepsilon $}-capacity of sets in functional space}, Amer. Math. Soc.
  Transl. (2) \textbf{17} (1961), 277--364. \MR{0124720}

\bibitem{MR0334381}
L.~Le~Cam, \emph{Convergence of estimates under dimensionality restrictions},
  Ann. Statist. \textbf{1} (1973), 38--53. \MR{0334381}

\bibitem{MR1899299}
Ji{\v{r}}{\'{\i}} Matou{\v{s}}ek, \emph{Lectures on discrete geometry},
  Graduate Texts in Mathematics, vol. 212, Springer-Verlag, New York, 2002.
  \MR{1899299 (2003f:52011)}

\bibitem{MR519520}
Robert Osserman, \emph{Bonnesen-style isoperimetric inequalities}, Amer. Math.
  Monthly \textbf{86} (1979), no.~1, 1--29. \MR{519520}

\bibitem{MR815583}
B.~L. Rothschild and E.~G. Straus, \emph{On triangulations of the convex hull
  of {$n$} points}, Combinatorica \textbf{5} (1985), no.~2, 167--179.
  \MR{815583 (87i:52021)}

\bibitem{MR1385671}
Aad~W. van~der Vaart and Jon~A. Wellner, \emph{Weak {C}onvergence and
  {E}mpirical {P}rocesses}, Springer Series in Statistics, Springer-Verlag, New
  York, 1996, With applications to statistics. \MR{1385671 (97g:60035)}

\bibitem{MR1742500}
Yuhong Yang and Andrew Barron, \emph{Information-theoretic determination of
  minimax rates of convergence}, Ann. Statist. \textbf{27} (1999), no.~5,
  1564--1599. \MR{1742500 (2001g:62006)}

\end{thebibliography}

\end{document}